\documentclass[11pt]{amsart}

\pdfoutput=1

\usepackage{amsmath}
\usepackage{fullpage}
\usepackage[shortlabels]{enumitem}
\usepackage{xspace}
\usepackage[psamsfonts]{amssymb}
\usepackage[utf8]{inputenc}
\usepackage{graphicx,color}
\usepackage[breaklinks=true]{hyperref}

\usepackage{graphicx}
\usepackage{amsmath}%
\usepackage{amsthm}%
\usepackage{amscd}
\usepackage{amsfonts}%
\usepackage{amssymb}%
\usepackage{graphicx}

\usepackage{mathrsfs}

\usepackage{tikz}
\usetikzlibrary{matrix,arrows}

\usepackage{tikz-cd}

\newtheorem{theorem}{Theorem}[section]

\newtheorem{corollary}[theorem]{Corollary}

\newtheorem{lemma}[theorem]{Lemma}

\newtheorem{proposition}[theorem]{Proposition}

\theoremstyle{remark}

\numberwithin{equation}{section}

\newcommand{\pfrak}{\mathfrak{p}}

\newcommand{\Pro}{\mathbb{P}}

\newcommand{\Z}{\mathbb{Z}}
\newcommand{\C}{\mathbb{C}}

\newcommand{\Q}{\mathbb{Q}}

\newcommand{\A}{\mathbb{A}}

\DeclareMathOperator{\rad}{rad}

\newcommand{\rk}{\mathrm{rank}\,}

\newcommand{\PP}{\mathbb{P}}
\DeclareMathOperator{\PSL}{PSL}
 
\DeclareMathOperator{\numnorm}{\mathbf{N}}

\newcommand\conjmax[1]%
{\begin{tikzpicture}[baseline=(A.base)]
	\node[inner sep=2pt] (A) {$#1$};
	\draw (A.south west) -- (A.north west) -- (A.north east) -- (A.south east);
\end{tikzpicture}%
}

\usepackage[numeric, initials, backrefs, sorted, nobysame]{amsrefs} 
\renewcommand{\PrintDOI}[1]{%
	\textsc{doi:} \href{https://doi.org/#1}{\ttfamily #1}%
}
\BibSpec{article}{%
	+{}{\PrintAuthors} {author}
	+{,}{ \textit} {title}
	+{,}{ } {journal}
	+{}{ \textbf} {volume}
	+{}{ \parenthesize} {date}
	+{,}{ } {pages}
	+{,}{ \PrintDOI} {doi}
	+{,}{ } {note}
	+{.}{} {transition}
	+{}{ } {review}
}

\makeatletter
\@namedef{subjclassname@2020}{%
  \textup{2020} Mathematics Subject Classification}
\makeatother

  \DeclareFontFamily{U}{wncy}{}
    \DeclareFontShape{U}{wncy}{m}{n}{<->wncyr10}{}
    \DeclareSymbolFont{mcy}{U}{wncy}{m}{n}
    \DeclareMathSymbol{\Sha}{\mathord}{mcy}{"58}

\begin{document}
\title[]{On the greatest prime factor of polynomial values and subexponential Szpiro in families}

\author{Jos\'e Cuevas Barrientos}
\address{ Departamento de Matem\'aticas,
Pontificia Universidad Cat\'olica de Chile.
Facultad de Matem\'aticas,
4860 Av.\ Vicu\~na Mackenna,
Macul, RM, Chile}
\email[J. Cuevas Barrientos]{josecuevasbtos@uc.cl}%

\author{Hector Pasten}
\address{ Departamento de Matem\'aticas,
Pontificia Universidad Cat\'olica de Chile.
Facultad de Matem\'aticas,
4860 Av.\ Vicu\~na Mackenna,
Macul, RM, Chile}
\email[H. Pasten]{hector.pasten@uc.cl}%

%%\thanks{}
\thanks{H.P. was supported by ANID Fondecyt Regular grant 1230507 from Chile.}
\date{\today}
% 11N32: Primes represented by polynomials; other multiplicative structures of polynomial values
% 11J86: Linear forms in logarithms; Baker’s method
% 11G05: Elliptic curves over global fields
% 11G18: Arithmetic aspects of modular and Shimura varieties


\begin{abstract}
	Combining a modular approach to the $abc$ conjecture developed by the second author with the classical method of
	linear forms in logarithms, we obtain improved unconditional bounds for two classical problems. First, for
	Szpiro's conjecture when the relevant elliptic curves are members of a one-parameter family (an elliptic
	surface). And secondly, for the problem of giving lower bounds for the greatest prime factor of polynomial
	values, in the case of quadratic and cubic polynomials. The latter extends earlier work by the second author for the polynomial $n^2+1$. 

\end{abstract}

\maketitle

%%\tableofcontents

%%%%%%%%%%%%%%%%%%%%%%%%%%%%%%%%%%%%%%
%%%%%%%%%%%%%%%%%%%%%%%%%%%%%%%%%%%%%%
%%%%%%%%%%%%%%%%%%%%%%%%%%%%%%%%%%%%%%
%%%%%%%%%%%%%%%%%%%%%%%%%%%%%%%%%%%%%%
%%%%%%%%%%%%%%%%%%%%%%%%%%%%%%%%%%%%%%
%%%%%%%%%%%%%%%%%%%%%%%%%%%%%%%%%%%%%%

\section{Introduction}

%%%%%%%%%
%%%%%%%%%
%%%%%%%%%

\subsection{Main results}
A common theme in the study of the multiplicative properties of polynomial values is the problem of giving lower bounds
for the radical and the prime divisors of such values. We focus on two concrete problems in this context: Szpiro's
conjecture (which concerns lower bound for the radical of the discriminant of elliptic curves) and giving lower bounds
for the greatest prime divisor of polynomial values.

Regarding Szpiro's conjecture, the strongest available bound comes from the theory of modular forms.
On the other hand, for the problem of the greatest prime factor of polynomials, the classical approach is based on the
theory of linear forms in logarithms, see sections \ref{sec:history_Szpiro} and  \ref{sec:history_P}.
Building on ideas of \cite{pasten2024n2+1} that concern the greatest prime factor of $n^2+1$ as well as some partial
progress on the $abc$ conjecture, we combine these two techniques in order to get stronger bounds in both problems. In
order to state our main results, let us introduce some notation.

For an elliptic curve $E$ over $\Q$ we let $h(E)$ be the Faltings height of $E$.
For a non-zero integer $m$, let $P(m)$ be the greatest prime factor of $m$. The symbol $\log_k^*(t)$ denotes the $k$-th
iterate of the logarithm whenever it is defined and takes a value $\ge 1$, otherwise we define it as $1$. 

Regarding Szpiro's conjecture, we prove:

\begin{theorem}\label{ThmMainSzpiro}
	Let $A,B\in \Z[t]$ be polynomials, coprime over $\Q$, not both constant, and such that the discriminant $D=-16(4A^3+27B^2)\in
	\Z[t]$ is not the zero polynomial.
	Consider the elliptic surface (on the parameter $t$) of affine Weierstrass equation 
	\begin{equation}
		E_t\colon\qquad y^2 = x^3 + A(t) x+ B(t).
		\label{EqnEllSurf}
	\end{equation}
	Let $\Sigma\subseteq \Z$ be the finite set of integers $n$ such that the fibre $E_n$ is not an elliptic curve.
	There is a constant $\kappa>0$ depending only on $A(t)$ and $B(t)$, such that for every
	$n\in\Z \setminus \Sigma$ one has
	\[
		\log |\Delta_n| \ll h (E_n)\le \exp\mathopen{}\left( \kappa \sqrt{(\log N_n)\log^*_2N_n} \right)\mathclose{},
	\]
	where $N_n$ is the conductor and $\Delta_n$ is the minimal discriminant of $E_n$.
\end{theorem}
\begin{corollary}[Subexponential bound for Szpiro's conjecture]\label{CoroSzpiro}
	With the notation of Theorem \ref{ThmMainSzpiro}, if $\Delta_n$ is the minimal discriminant of $E_n$ for $n\in
	\Z\setminus\Sigma$, then for every $\epsilon>0$ one has
	\[
		\log\lvert\Delta_n\rvert \ll h (E_n) \ll_\epsilon N_n^\epsilon.
	\]
\end{corollary}

On the other hand, regarding the largest prime factor of polynomial values, we prove:

\begin{theorem}\label{ThmMainP}
	Let $f(x) \in \Z[x]$ be a quadratic polynomial with two complex roots or a cubic of the form $f(x) = (ax + b)^3 + c$ with $a,c\ne 0$.
Then we have 
$$
\rad(f(n)) \ge \exp\left(\kappa_f \cdot\frac{(\log^*_2n)^2}{\log^*_3 n}\right)
$$
for certain constant $\kappa_f>0$ depending only on $f$, and
$$
P(f(n)) \gg_f \frac{(\log^*_2n)^2}{\log^*_3 n}.
$$
\end{theorem}

Here, as usual, for a non-zero integer $m$ we let its radical $\rad(m)$ be the largest positive squarefree divisor of $m$. In the next subsections we provide further discussion on the context of these results.
%%%%%%%%%
%%%%%%%%%
%%%%%%%%%

\subsection{Szpiro's conjecture}
\label{sec:history_Szpiro}
It is a conjecture by Szpiro \cite{szpiroDC} that there is a constant $c > 0$ such that for every elliptic curve $E$ over $\Q$ with minimal discriminant $\Delta_E$ and conductor $N_E$, one has
$\lvert\Delta_E\rvert \le N_E^c$. This problem remains open and it is deep; for instance, it implies a form of the $abc$ conjecture.

There is also a related problem known as the \emph{height conjecture} \cite{frey1989links}.
This conjecture predicts that for every elliptic curve $E$ over $\Q$ one has $h(E)\ll \log N_E$. The height conjecture implies Szpiro's conjecture because $\log |\Delta_E|\ll h(E)$, see
\cite{frey1989links,silverman86heights}.
Conversely, Szpiro's conjecture implies a form of the $abc$, and the $abc$ conjecture implies the height conjecture (see
\cite{murty2013modular} for further discussion).

At present, the strongest unconditional results towards the height conjecture and Szpiro's conjecture are due to Murty and the second author \cite{murty2013modular}: the result is the exponential bound $h(E)\ll N_E\log N_E$.

For some restricted families of elliptic curves one has stronger bounds.
For instance, the results in \cite{stewart2001abc} imply that when $E$ is the Frey curve attached to an $abc$ triple one
has $h(E)\ll_\epsilon N_E^{1/3 +\epsilon}$, but this bound is exponential too.
Theorem 1.4 in \cite{pasten2024n2+1} implies a \emph{subexponential bound}; namely, there is a constant $\kappa>0$ such that
\[
	h(E) \ll_\eta \exp\mathopen{}\left( \kappa \sqrt{(\log N_E)\log_2 N_E} \right)\mathclose{},
\]
for Frey curves $E$ attached to $abc$ triples with $a \le c^{1 - \eta}$ for any fixed $\eta > 0$. Theorem \ref{ThmMainSzpiro} establishes the same bound for elliptic curves varying in a one-parameter family, that is, for fibres of an elliptic surface.

%%%%%%%%%
%%%%%%%%%
%%%%%%%%%

\subsection{The greatest prime factor}
\label{sec:history_P}
It is a classical problem to study the growth of $P(f(n))$ for a given polynomial $f(x)\in \Z[x]$.
This theme can be traced back at least to the works of St\"ormer \cite{stormer97pell} and \cite{stormer98indeterminee}
in the late 19th century.
Mahler \cite{mahler33grossten} in 1933 and Chowla \cite{chowla34largest} in 1934 proved
\begin{equation}
	P(f(n)) \gg \log^*_2 n
	\label{eqn:generalold_ineq}
\end{equation}
for the polynomial $f(n)=n^2+1$. This bound was later established for quadratic polynomials $f(n)$ by Schinzel \cite{schinzel67gelfond},
and for cubic polynomials by Keates \cite{keates68greatest}, as long as $f$ has at least two complex roots. After a sequence of further articles, the bound \eqref{eqn:generalold_ineq} was established for all $f(x)\in \Z[x]$ with at least two complex roots, see \cite{shorey76greatest} and the references therein.

%%%

In 2001, Stewart and Yu  obtained the following bound
for \emph{reducible} quadratic polynomials $f(x)$ (see the equation (7) in page~171 of \cite{stewart2001abc})
\begin{equation}
	P(f(n)) \gg (\log^*_2 n) \frac{\log^*_3n}{\log^*_4n}.
	\label{eqn:stewart_yu}
\end{equation}
A similar bound was obtained by Haristoy \cite{haristoy2003exponentielles} in 2003 for certain binary forms which, in particular, implies \eqref{eqn:stewart_yu} for all quadratic polynomials with two complex roots. Finally, in 2006 the bound \eqref{eqn:stewart_yu} was proved in full generality by Gy\H ory and Yu, see Cor. 5 in \cite{gyory2006form_eq}.

Theorem \ref{ThmMainP} gives a substantial improvement in the quadratic and cubic cases.

This result is a generalization of the bounds obtained in \cite{pasten2024n2+1} by the second author, where the same
estimate is proved for the polynomial $f(n)=n^2+1$ as well as for all reducible quadratic polynomials with two complex
roots (see Cor.~1.5 in \cite{pasten2024n2+1}).
Therefore, the new results in the quadratic case concern irreducible polynomials.

In a different direction, many authors have considered the problem of giving lower bounds for $P(f(n))$ for at least one
value of $n$ on an interval. In this case better bounds are obtained, see for instance \cite{shorey76greatest}. We note,
however, that the problem that we consider here concerns bounds valid for \emph{every} sufficiently large value of $n$.
This makes a difference as, in practice, it can happen that $P(f(n_0))$ is small for a certain $n_0$ although for
neighboring values of $n$ one has that $P(f(n))$ is large.
For instance, the following example was found by Luca \cite{luca2004primitive}:
\[\begin{array}{cr}
	n & P(n^2 + 1) \\
	\hline
	24\,208\,141 & 119\,529\,857 \\
	24\,208\,142 & 121\,140\,377 \\
	24\,208\,143 & 67\,749\,617\,053 \\
	24\,208\,144 & 89 \\
	24\,208\,145 & 5\,218\,192\,121 \\
	24\,208\,146 & 586\,034\,332\,757\,317 \\
	24\,208\,147 & 58\,603\,438\,117\,361 \\
	24\,208\,148 & 117\,206\,885\,917\,981 \\
	24\,208\,149 & 2\,292\,977\,009 \\
	24\,208\,150 & 127\,793\,609 \\
\end{array}\]

Previous works on lower bounds for $P(f(n))$ mainly used linear forms in complex and $p$-adic  logarithms as explained,
for instance, in section 9.6 of \cite{evertse:unit}.
Instead, we follow an approach inspired in \cite{pasten2024n2+1} which takes as a new input estimates from
\cite{pasten2023shimura} and \cite{murty2013modular} that build on the theory of Shimura curves and modularity of
elliptic curves.
At a technical level, our new contribution has two main aspects: on the one hand, we construct new elliptic curves
tailored to the particular polynomials under consideration, and on the other hand, we introduce to the problem some
tools from algebraic number theory to deal with number fields with more complicated arithmetic --- in
\cite{pasten2024n2+1}, only the field $\Q(i)$ is necessary.
We mention that, unlike \cite{pasten2024n2+1}, we are able to deal with the case of cubic polynomials, not only
quadratic. Leaving open the possibility for further work, we proved the criterion given below in
Theorem~\ref{thm:main_criteria} in more generality (not only the quadratic and cubic cases) as a tool for obtaining
improved lower bounds on $\rad(f(n))$ and $P(f(n))$.

%%%%%%%%%
%%%%%%%%%
%%%%%%%%%

\subsection{The methods}
%%%
%%%

For a prime $p$ let $\nu_p$ be the $p$-adic valuation on $\Z$.  As a key tool for proving our main results, we establish the following general criterion for obtaining lower bounds for radicals of polynomial values:

\begin{theorem}\label{thm:main_criteria}
	Let $F(x) \in \Z[x]$ be a polynomial with at least two different complex roots that satisfies the following property:
	there exists a constant $\mu := \mu(F) > 0$ such that for all but finitely many $n \in \Z$ we have
	\begin{equation}
		\prod_{p \mid F(n)} \nu_p(F(n)) \ll_F \rad(F(n))^\mu.
		\label{eqn:main_thm_condition}
	\end{equation}
	Then there exists a constant $\kappa := \kappa(F) > 0$ such that for all $n\in \Z$ we have
	\[
		\log |n| \le \exp\mathopen{}\big( \kappa \sqrt{(\log^*\rad F(n))\log^*_2\rad F(n)} \big)\mathclose{}.
	\]
\end{theorem}

Given an elliptic curve $E$ over $\Q$ we let $D_E$ and $N_E$ be the absolute value of its minimal discriminant and the conductor of $E$ respectively.  In order to deduce Theorem \ref{ThmMainSzpiro} we verify Condition \eqref{eqn:main_thm_condition} using the following results coming from the theory of modular forms and Shimura curves respectively.

\begin{theorem}[Thm.~7.1 in \cite{murty2013modular}]\label{thm:additive_bound}
	There is an absolute and effective constant $\kappa > 0$ with the following property:  for each elliptic curve $E$ over $\Q$ one has
	\[
		\log(D_E) \le \kappa \cdot N_E \log N_E.
	\]
\end{theorem}

\begin{theorem}[Cor.~16.3 in \cite{pasten2023shimura}]\label{thm:semistable_bound}
	Let $S$ be a finite set of prime numbers and let $\epsilon > 0$ be a positive real number.
	There is a constant $\kappa := \kappa(S, \epsilon) > 0$ with the following property:

	For each elliptic curve  $E$ over $\Q$ which is semistable outside of $S$ we have
	\[
		\prod_{\substack{p \mid N_E \\ p\notin S}} \nu_p(D_E) \le \kappa \cdot N_E^{11/2 + \epsilon}.
	\]
\end{theorem}

%%%%%%%%%%%%%%%%%%%%%%%%%%%%%%%

Regarding Theorem \ref{ThmMainP}, we will see that it follows from Theorem \ref{ThmMainSzpiro} after a suitable construction of elliptic curves.
%%%%

%%%%

%%%%%%%%%%%%%%%%%%%%%%%%%%%%%%%%%%%%%%
%%%%%%%%%%%%%%%%%%%%%%%%%%%%%%%%%%%%%%
%%%%%%%%%%%%%%%%%%%%%%%%%%%%%%%%%%%%%%
%%%%%%%%%%%%%%%%%%%%%%%%%%%%%%%%%%%%%%
%%%%%%%%%%%%%%%%%%%%%%%%%%%%%%%%%%%%%%
%%%%%%%%%%%%%%%%%%%%%%%%%%%%%%%%%%%%%%

\section{Preliminaries} 

%%%%%%%%
%%%%%%%%
%%%%%%%%

\subsection{Linear forms in logarithms} Let $K$ be a number field of degree $g$ and for $\alpha\in K$ let $h(\alpha)$ be its height normalized to $\Q$. Given a nonzero ideal $\mathfrak{a} \subseteq \mathcal{O}_K$, we define the (absolute) norm
as $\numnorm\mathfrak{a} = |\mathcal{O}_K/\mathfrak{a}|$.
For an infinite place $v \in M_K^\infty$, we define $\numnorm(v) := 2$,
and for a finite place $v \in M_K^0$ induced by a prime ideal $\mathfrak{p}$,
we denote $\numnorm(v) := \numnorm\mathfrak{p}$. We need the following result coming from the theory of linear forms in logarithms, which builds on earlier work by Matveev and Yu (see Thm.~4.2.1 in \cite{evertse:unit}).

\begin{theorem}\label{thm:linforms}
	Let $K$ be a number field,
	$\Gamma \le K^\times$ be a finitely generated multiplicative subgroup
	and $\xi_1, \dots, \xi_m$ be a system of generators for $\Gamma/\Gamma_{\rm tors}$.
	Then, for any place $v \in M_K$ and every $\xi \in \Gamma-\{1\}$ we have
	\[
		-\log\mathopen{}\lvert1 - \xi\rvert_v
		< \kappa^m \cdot \frac{\numnorm(v)}{\log\numnorm(v)}
		\log\max\left\{ e, \numnorm(v)h(\xi) \right\} \prod_{j=1}^{m} h(\xi_j),
	\]
	where $\kappa$ is an effectively computable constant depending only on $g=[K:\Q]$.
\end{theorem}

%%%%%%%%
%%%%%%%%
%%%%%%%%

\subsection{Algebraic integers of controlled height}
As a technical tool for the proof of Theorem~\ref{thm:main_criteria}, we will need the following proposition to compare the height of an element in an ideal to its norm.

\begin{proposition}\label{thm:main_bound}
	There is an effectively computable constant $c > 0$ (depending only on  $K$) with the following property:
	Given any nonzero ideal $\mathfrak{a} \subseteq \mathcal{O}_K$ there exists an element $\gamma \in \mathfrak{a} \setminus \{ 0 \}$ such that
	\[
		h(\gamma) \le c \log\numnorm\mathfrak{a}.
	\]
	Moreover, if $\mathfrak{a}$ is principal, we may choose $\gamma$ to be a generator of $\mathfrak{a}$.
\end{proposition}
\begin{proof}
By geometry of numbers, we have such a bound with $h(\gamma)$ replaced by the norm of $\gamma$ (see for instance \cite{lang:algebraic}). Then the result follows from Lemma~3 of \cite{gyory2006form_eq} (or Prop.~4.3.12 of \cite{evertse:unit}).
\end{proof}

%%%%%%%%%%%%%%%%%%%%%%%%%%%%%%%%%%%%%%
%%%%%%%%%%%%%%%%%%%%%%%%%%%%%%%%%%%%%%
%%%%%%%%%%%%%%%%%%%%%%%%%%%%%%%%%%%%%%
%%%%%%%%%%%%%%%%%%%%%%%%%%%%%%%%%%%%%%
%%%%%%%%%%%%%%%%%%%%%%%%%%%%%%%%%%%%%%
%%%%%%%%%%%%%%%%%%%%%%%%%%%%%%%%%%%%%%

\section{The main criterion: lower bounds for the radical} 
In this section we prove Theorem~\ref{thm:main_criteria}.

Let $\alpha_1 t - \beta_1$ and $\alpha_2 t - \beta_2$ be two non-proportial linear factors of $F(t)$ with $\alpha_j,
\beta_j$ algebraic integers.
Let $K := \Q(\alpha_1, \alpha_2, \beta_1, \beta_2)$ and let $g := [K : \Q]$. We define the polynomial
\[
	G(t) := (\alpha_1 t - \beta_1) (\alpha_2 t - \beta_2) \in \mathcal{O}_K[t].
\]
In this section, $\kappa_1,\kappa_2,...$ will denote some auxiliary positive constants depending only on $F$. 

From now on, $n$ denotes an integer in $\Z$ and the constants $\kappa_j$ are independent of $n$. We also define
\[
	\omega := \alpha_2(\alpha_1n - \beta_1) - \alpha_1(\alpha_2n - \beta_2) = \alpha_1\beta_2 - \alpha_2\beta_1,
\]
which is a non-zero element of $O_K$ independent of $n$.
%%%%%%%%
%%%%%%%%
%%%%%%%%

\subsection{From ideals to integers} Let $\pfrak_1,...,\pfrak_r$ be the prime ideals dividing $\alpha_1 \alpha_2 G(n)$. Then, as fractional ideals, we have
$$
\left(\frac{\alpha_1(\alpha_2n - \beta_2)}{\alpha_2(\alpha_1n - \beta_1)}\right)= \prod_{j=1}^r \pfrak^{e_j}
$$
for certain integer exponents $e_j$. Let $\mathbf{h}$ be the class number of $K$ and for each $j$, we choose a generator $\xi_j$ of $\mathfrak{p}_j^{\mathbf{h}}$ satisfying the bound in Proposition~\ref{thm:main_bound}. Then we get
$$
(\xi)= \prod_{j=1}^r (\xi_j)^{e_j}
$$
where
$$
\xi:=\left(\frac{\alpha_1(\alpha_2n - \beta_2)}{\alpha_2(\alpha_1n - \beta_1)}\right)^{\mathbf{h}}.
$$
Therefore, there is a unit $u \in \mathcal{O}_K^\times$ such that $\xi = u \xi_1^{e_1} \cdots
\xi_r^{e_r}$.

Let $\{ \eta_1, \dots, \eta_s \}$ be a basis for the units $\mathcal{O}_K^\times$ modulo torsion, where $s = \rk(\mathcal{O}_K^\times)$. Let $J := \{ 1, \dots, r \}$ and $J' := \{ 1, \dots, s \}$. Then we have
\[
	\xi  = \left( 1 - \frac{\omega}{\alpha_2(\alpha_1n - \beta_1)} \right)^{\mathbf{h}}	= \zeta \cdot \prod_{i \in J^\prime} \eta_i^{e'_i} \cdot \prod_{j\in J} \xi_j^{e_j},
\]
where $\zeta \in K^\times$ is a root of unity and the $e'_i$ are the integer exponents.

%%%%%%%%
%%%%%%%%
%%%%%%%%

\subsection{Choosing factors}
Set
\[
	B := \exp\mathopen{}\left( \sqrt{(\log^* R)\log^*_2 R} \right),
\]
where $R := \rad F(n)$.
Let $I := \{ j \in J : |e_j| > B \}$ and define
\[
	\xi_0 := \prod_{j\in J- I} \xi_j^{e_j}.
\]
Set $I_0 := I \cup \{ 0 \}$ and $m := |I_0|$.

Pick any archimedian absolute value $|-|$ on $K$.  As $\omega$, $\alpha_1$, $\alpha_2$ and $\beta_1$ are all fixed, all the chosen absolute value of  $q := \kappa_1/(\alpha_2(\alpha_1n - \beta_1))$ is less than $\kappa_1/|n|$ as $|n|$ grows, and therefore 
\[
	|1 - \xi| = |1 - (1 - q)^{\mathbf{h}}| <\frac{\kappa_2}{|n|},\quad n\ne 0.
\]

Consider the subgroup $\Gamma \le K^\times$ generated by the roots of unity in $K$, the element $\xi_0$, and the elements $\eta_i$ for $i\in J'$ and $\xi_j$ for $j\in I$. By Theorem~\ref{thm:linforms}, as $|n|$ grows we have 
\begin{equation}
	\begin{aligned}
		\kappa_3\log |n| < \log |n| - \log \kappa_2
		&\le -\log |1-\xi|\\
		&\le \kappa_3^{m+s} \left(\log^*h(\xi)\right) \prod_{j \in J'} h(\eta_j)\cdot\prod_{i \in I_0} h(\xi_i)\\
		&\le \kappa_4^{m} \left(\log^*h(\xi)\right) \prod_{i \in I_0} h(\xi_i)
	\end{aligned}
	\label{eqn:gyorys_bound}
\end{equation}
because the choice of the elements $\eta_1,...,\eta_s$ only depend on $K$, and $m\ge 1$. 

Notice that as $|n|$ grows
$$
	h(\xi) = {\mathbf{h}} \cdot h\left( \frac{\alpha_1(\alpha_2n - \beta_2)} {\alpha_2(\alpha_1n - \beta_1)} \right) \le \kappa _5 \log |n|
$$
and we deduce
$$
\kappa_3\log |n| \le \kappa_6^{m} (\log^*_2|n|)\prod_{i \in I_0} h(\xi_i).
$$
Hence,
$$
\frac{\log |n|}{\log^*_2|n|} \le \kappa_7^{m} \prod_{i \in I_0} h(\xi_i).
$$

%%%%%%%%
%%%%%%%%
%%%%%%%%

\subsection{Proof of the lower bound for the radical}

Continuing with the previous argument, in this section we conclude the proof of Theorem  \ref{thm:main_criteria}.

Let $p_j$ be the prime number below $\pfrak$. By choice of the $\xi_j$ we have
$$
h(\xi_j) \le \kappa_8 \log p_j
$$
 for $j=1,...,r$. Regarding $h(\xi_0)$, we have
$$
h(\xi_0) \le  \sum_{j\in J- I} h(\xi_j^{e_j}) \le B \sum_{j\in J- I} h(\xi_j) \le \kappa_8 B \log R
$$
where we recall that $R=\rad F(n)$. Hence
\begin{equation}\label{EqnKeyBound}
\frac{\log |n|}{\log^*_2|n|} \le \kappa_9^{m} B(\log R) \prod_{i\in I} \log p_i.
\end{equation}

If $I = \emptyset$ so that $m=1$, then
$$
\frac{\log |n|}{\log^*_2|n|} \le \kappa_{10} B(\log R) 
$$
which gives $\log |n| \le B^{\kappa_{11}}$ as desired.  So we may assume $m\ge 2$, that is, $I$ non-empty.

The arithmetic-geometric mean inequality gives
$$
\prod_{i\in I} \log p_i \le \left(\frac{\log R}{ m-1}\right)^{m-1}\le (\log R)^{m-1}
$$
and from \eqref{EqnKeyBound} we get
\begin{equation}\label{EqnMainLast}
\frac{\log |n|}{\log^*_2|n|} \le \kappa_{12}^{m} B(\log R)^m.
\end{equation}

From assumption \eqref{eqn:main_thm_condition} and the definition of the index set $I$ (whose cardinality is $m-1$), we deduce that 
$$
(\kappa_{13}B)^{m-1} \le R^{\kappa_{14}}
$$
from which it follows that
$$
m \le \kappa_{15}\frac{\log^* R}{ \log^* B} = \kappa_{15} \sqrt{\frac{\log^* R}{\log^*_2 R}}.
$$

Using this in \eqref{EqnMainLast} we get
$$
\sqrt{\log|n|}\le \frac{\log |n|}{\log^*_2|n|} \le B \exp\left(\kappa_{16} (\log^*_2 R)  \sqrt{\frac{\log^* R}{\log^*_2 R}}\right) \le B^{\kappa_{17}}
$$
which concludes the  proof of Theorem \ref{thm:main_criteria}.

%%%%%%%%%%%%%%%%%%%%%%%%%%%%%%%%%%%%%%
%%%%%%%%%%%%%%%%%%%%%%%%%%%%%%%%%%%%%%
%%%%%%%%%%%%%%%%%%%%%%%%%%%%%%%%%%%%%%
%%%%%%%%%%%%%%%%%%%%%%%%%%%%%%%%%%%%%%
%%%%%%%%%%%%%%%%%%%%%%%%%%%%%%%%%%%%%%
%%%%%%%%%%%%%%%%%%%%%%%%%%%%%%%%%%%%%%

\section{Bounds for Szpiro's conjecture}
\label{sec:szpiro}

The following result is well-known, but we could not find a convenient reference.
\begin{lemma}\label{LemmaY}
	Let $\pi \colon X \to \PP^1_\C$ be a non-isotrivial elliptic surface over $\C$.
	Then it has at least three bad fibres.
\end{lemma}
\begin{proof}
	Let $T \subseteq \PP^1_\C$ be the points supporting the bad fibres and let $S := \PP^1_\C \setminus T$.
	As the $j$-invariant is non-constant, $\# T \ge 1$.
	Suppose, by contradiction, that $\# T \le 2$, then the (topological) universal cover of $S$ is $\C$.

	Let $\Gamma \le \PSL_2(\Z)$ be a congruence subgroup and let $Y_\Gamma$ be the corresponding open modular
	curve over $\C$.
	Taking $\Gamma$ sufficiently small we may assume that $Y_\Gamma$ has no elliptic points and, thus, it is
	hyperbolic.
	There is an étale cover $S' \to S$ with a classifying map $S' \to Y_\Gamma$ which is non-constant as the
	elliptic surface is non-isotrivial.
	But since $\C$ is the universal cover of $S$, we get a non-constant map $\C \to S'$, which contradicts the
	hyperbolicity of $Y_\Gamma$.
\end{proof}

\begin{lemma}\label{LemmaRH}
	Let $f, g \in \C[t]$ be coprime, not both constant.
	Then $f^2 + g^3$ is non-constant.
\end{lemma}
\begin{proof}
	For otherwise, we would get a non-constant morphism $\PP^1_\C \to E$, where $E$ is the complex elliptic curve
	given by $y^2 = x^3 + 1$.
	This is not possible, by the Riemann-Hurwitz formula.
\end{proof}

Finally, we can prove our subexponential bounds for the height conjecture. 

\begin{proof}[Proof of Theorem \ref{ThmMainSzpiro}]
	By Lemma \ref{LemmaRH} the discriminant $D(t)$ is non-constant. 

	Let $\rho\in \Z$ be the resultant of $A(t)$ and $B(t)$. Since $A,B$ are coprime over $\Q$, we have $\rho\ne 0$. For every $n\in \Z$ we get that $\gcd(A(n),B(n))$ divides $\rho$, so the equation defining $E_n$ is \emph{quasiminimal}: there is a constant $\tau \ge 1$ independent of $n$ such that $D(n)$ divides $\tau |\Delta_n|$. Furthermore, $E_n$ is semistable outside the primes dividing $\rho$; in particular $R(n)$ divides $\rho N_n$ where $R(n)=\rad (D(n))$.

	By coprimality of $A(t)$ and $B(t)$, Tate's
	algorithm gives at least one fibre of multiplicative reduction (a zero of $D(t)$). Thus, the elliptic surface
	defined by \eqref{EqnEllSurf} is non-isotrivial and it has at least two bad fibres in the affine chart $t\in
	\A^1\subseteq \Pro^1$, by Lemma \ref{LemmaY}. This means that $D(t)\in \Z[t]$ has at least two different complex
	zeros and Theorem \ref{thm:main_criteria} gives
	\[
		\log |n| \le \exp\mathopen{}\left(\kappa \sqrt{(\log R(n)) \log^*_2R(n)} \right)\mathclose{},
	\]
which applies thanks to Theorems \ref{thm:semistable_bound} and \ref{thm:additive_bound} (the second one for the primes of additive reduction of $E_n$); here we used the fact that $D(n)$ divides $\tau |\Delta_n|$.
	
	 As $R(n)$ divides $\rho N_n$ and
	\[
		h(E_n) \asymp \log \max\{|A(n)|^3,|B(n)|^2\} \asymp \log |n|
	\]
	(see \cite{silverman86heights}),  the result follows.
\end{proof}

%%%%%%%%%%%%%%%%%%%%%%%%%%%%%%%%%%%%%%
%%%%%%%%%%%%%%%%%%%%%%%%%%%%%%%%%%%%%%
%%%%%%%%%%%%%%%%%%%%%%%%%%%%%%%%%%%%%%
%%%%%%%%%%%%%%%%%%%%%%%%%%%%%%%%%%%%%%
%%%%%%%%%%%%%%%%%%%%%%%%%%%%%%%%%%%%%%
%%%%%%%%%%%%%%%%%%%%%%%%%%%%%%%%%%%%%%

\section{The greatest prime factor of polynomial values} 

In this section we prove Theorem \ref{ThmMainP}.
%%%%%%%%
%%%%%%%%
%%%%%%%%

\subsection{Quadratic case}
Let $f(x) = ax^2 + bx + c \in \Z[x]$ be a quadratic polynomial. Let $\delta := b^2-4ac$ be its discriminant. Consider the  curve $E$ defined by the Weierstrass equation
\begin{equation}
	E\colon \quad y^2 = x^3 - 3 \delta x - 2 \delta (2an + b).
	\label{eqn:elliptic_quadr}
\end{equation}

Theorem \ref{ThmMainP} in the quadratic case follows from Theorem \ref{ThmMainSzpiro} with $A(n)=-3\delta$ and $B(n)= 2\delta(2an+b)$. Indeed, with these choices one has
$$
D(n)=-6912\delta^2af(n)
$$
and noticing that the Weierstrass equation is quasi-minimal and semistable at each $p\nmid 6\delta$ (see \cite{tate75algorithm}) we deduce from Theorem \ref{ThmMainSzpiro} that
$$
\log n \ll \log |D(n)| \ll \exp\left(\kappa\sqrt{(\log^*R(n))\log^*_2R(n)}\right)
$$
where $R(n)=\rad(f(n))$. This proves the desired lower bound for $R(n)$. 

By classical estimates on prime numbers, for every integer $M>1$ one has
$$
4^{P(M)} \ge \prod_{p\le P(M)} p \ge \rad(M)
$$
and the lower bound for $P(f(n))$ follows.

%%%%%%%%
%%%%%%%%
%%%%%%%%

\subsection{Cubic case}
The cubic case of Theorem \ref{ThmMainP} is proved exactly as the quadratic case, using the family of elliptic curves
$$
E_n: \quad y^2 = x^3 + 3c(an+b)x +2c^2.
$$
Here one chooses $A(n)=3c(an+b)$ and $B(n)=2c^2$, which gives the discriminant
$$
D(n)= -1728 c^3\cdot f(n)
$$
where $f(x)=(ax+b)^3+c$. Then one concludes by applying Theorem \ref{ThmMainSzpiro}.
%%%%%%%%%%%%%%%%%%%%%%%%%%%%%%%%%%%%%%
%%%%%%%%%%%%%%%%%%%%%%%%%%%%%%%%%%%%%%
%%%%%%%%%%%%%%%%%%%%%%%%%%%%%%%%%%%%%%
%%%%%%%%%%%%%%%%%%%%%%%%%%%%%%%%%%%%%%
%%%%%%%%%%%%%%%%%%%%%%%%%%%%%%%%%%%%%%
%%%%%%%%%%%%%%%%%%%%%%%%%%%%%%%%%%%%%%
%%%%%%%%%%%%%%%%%%%%%%%%%%%%%%%%%%%%%%
%%%%%%%%%%%%%%%%%%%%%%%%%%%%%%%%%%%%%%
%%%%%%%%%%%%%%%%%%%%%%%%%%%%%%%%%%%%%%
%%%%%%%%%%%%%%%%%%%%%%%%%%%%%%%%%%%%%%
%%%%%%%%%%%%%%%%%%%%%%%%%%%%%%%%%%%%%%

\section*{Acknowledgments}
H.P. was supported by ANID Fondecyt Regular grant 1230507 from Chile. 
We are grateful to Cameron L.~Stewart for suggesting the generalization to quadratic polynomials, and to K\'alm\'an Gy\"ory for comments on a first version of Section~\ref{sec:szpiro}.

%%%%%%%%%%%%%%%%%%%%%%%%%%%%%%%%%%%%%%
%%%%%%%%%%%%%%%%%%%%%%%%%%%%%%%%%%%%%%
%%%%%%%%%%%%%%%%%%%%%%%%%%%%%%%%%%%%%%

%\begin{bibdiv}
%\begin{biblist}

\end{document}